\documentclass[a4paper]{amsart}
\usepackage{color}

\usepackage{enumitem}

\usepackage[utf8]{inputenc}

\usepackage{amsmath}
\usepackage{amsfonts}
\usepackage[capitalise]{cleveref}

\usepackage{graphicx}
\usepackage{float}
\usepackage{subfigure}
\usepackage{cleveref}

\newtheorem{theorem}{Theorem}[section]
\newtheorem{corollary}[theorem]{Corollary}
\newtheorem{lemma}[theorem]{Lemma}
\newtheorem{proposition}[theorem]{Proposition}

\newtheorem{rmk}[theorem]{Remark}
\newtheorem{remark}[theorem]{Remark}
\newtheorem{example}[theorem]{Example}

\def\r{\mathbb{R}}

\begin{document}

\title{Lower bound results for conditionally decomposable polytopes}

\author{Jie Wang \and David Yost}

\maketitle

\begin{abstract}
It is possible for a combinatorial type of polytope to have
both decomposable and indecomposable realizations; here decomposability is meant with respect to Minkowski addition. Such polytopes are called conditionally decomposable. We show that the minimum number of vertices of a conditionally decomposable $d$-polytope is in the range $[3d-3, 4d-4]$, and that for a polytope having a line segment for a summand, $4d-4$ is sharp. As an application, the exact lower bound of the number of $k$-faces of a decomposable $d$-polytope with $2d+m$ ($2 \le m\le d-4$) vertices is obtained.
Concerning the facets, in dimension 4, the minimum number of facets of a conditionally decomposable polytope is 9, and in dimension $d\ge 5$, the minimum is $d+4$. 

\end{abstract}

\section{Introduction}


A polytope is the convex hull of finite number of points in Euclidean space. Recall that a polytope $P$ is decomposable if it can be written as the Minkowski sum of two polytopes, i.e. $P=Q+R=\{q+r: q \in Q, r \in R \}$, where $Q, R$ are not homothetic to $P$, i.e. not of the form $kP+a$ for $k>0$, $a\in\r^d$. The set of all faces of a polytope is a  lattice with respect to inclusion. Two polytopes are called combinatorially equivalent if there is an isomorphism between their face lattices. For general information about polytopes, we refer to \cite{Grunbaum2003}.


It has long been known, at least in three dimensions, that the combinatorial structure of a polytope does not determine its decomposability. Thus it is natural to call a polytope {\it combinatorially decomposable} if every polytope combinatorially equivalent to it is decomposable, {\it combinatorially indecomposable} if every polytope combinatorially equivalent to it is indecomposable, and {\it conditionally decomposable} if one polytope combinatorially equivalent to it is decomposable and another one combinatorially equivalent to it is indecomposable. The concept of decomposability of polytopes was introduced by Gale \cite{Gale1954} (with a different name); Shephard \cite{Shephard1963} then made a serious study of it.

The existence of conditionally decomposable 3-polytopes was first given by Meyer \cite[Section 3.1]{Meyerthesis}, that the cuboctahedron is conditionally decomposable, followed by Kallay \cite{Kallay1982} who extended the study of the decomposability of polytopes to the study of the decomposability of geometric graphs, and found an example with 10 vertices. 



Smilansky \cite{Smilansky1987} made a groundbreaking study of decomposability of polytopes, establishing \cite[Section 6]{Smilansky1987} that conditionally decomposable $3$-polytopes with $V$ vertices and $F$ facets exist only if
$V \le F \le 2V-8$. There he also announced the converse, i.e. that there is a conditionally decomposable $3$-polytope with $V$ vertices and $F$ facets whenever $F$ and $V$ satisfy this inequality, but until recently the proof was only available in \cite[Chapter 8]{Smilanskythesis}. For a sketch of this proof, see \cite[Problem 6.5.15]{Pineda-2023}. It is easy to see that for $3$-polytopes, the conditionally decomposable polytopes with $8$ vertices (and $8$ facets) are the ones with the minimum number of vertices. Those examples in Smilansky's thesis all have a line segment as a summand. He effectively gave two conditionally decomposable examples with $V=F=8$; these turn out to be the examples numbered 282 and 288 in the catalogue due to Federico \cite{Federico1975}, of all $3$-polytopes with up to $8$ faces. (We will refer henceforth to the $n^{th}$ example in this catalogue as $Fn$.) These two examples were rediscovered in \cite{Przeslawski-Yost-2008} and \cite{Yost2007}.  

We need to know some sufficient conditions for the indecomposability of polytopes. Rather than give a detailed account, we just summarize the results which we need. Kallay \cite{Kallay1982} defined a {\it geometric graph} as any graph whose vertices lie in some Euclidean space, and whose edges are some of the segments between them. This includes the edge graph, i.e. the $1$-skeleton, of any polytope. He then generalized to such graphs the notion of decomposability of polytopes. We will say that a geometric graph $(V,E)$ is decomposable if there is a {\it decomposing function} $\phi:V\to\r^d$ which is not the restriction of a homothety on $\r^d$. By decomposing function we mean any function on the vertex set which respects the direction of edges, i.e. for any edge $e=[p,q]\in E$, there  is a scalar $\lambda_e$  such that $\phi(p)-\phi(q)=\lambda_e (p-q)$. Note that the scalar does not necessarily have to be positive. This generalisation is consistent, i.e. a polytope $P$ is decomposable if and only if its $1$-skeleton $G(V(P), E(P))$ is decomposable \cite[Corollary 5]{Kallay1982}. 

By simple extension of a geometric graph $G$ we mean a graph whose vertices are are all the vertices of $G$ and one new vertex $v$, and whose edges are all the edges of $G$ and two new edges containing $v$.
By {\it strongly connected triangular chain}, we mean a finite sequence of triangles $T_1,\cdots,T_n$ such that $T_i\cap T_{i+1}$ is an edge for each $i$. By {\it triangle} in a geometric graph, we mean three vertices, each pair of which determines an edge of the graph. In the graph of a polytope, a triangle may be, but need not be, a triangular face. For more information on the decomposability of polytopes, we refer to \cite{Schneider} and \cite{Grunbaum2003}. The following omnibus result combines work from \cite{Gale1954}, \cite[Theorem 8]{Przeslawski-Yost-2008} and \cite[Corollary 5]{Kallay1982}.

\begin{theorem}\label{theorem-indecomposable}
	\begin{enumerate}
		\item Any pyramid is indecomposable.	
		\item If the graph of a $d$-polytope $P$ contains an indecomposable subgraph, whose union touches every facet, then $P$ is indecomposable.
		\item Let $C$ be the disjoint union of indecomposable geometric graphs $A$ and $B$ and two disjoint edges $[a_1, b_1]$, $[a_2, b_2]$ with $a_i \in A$, $b_i \in B$, $i=1,2$.
		If the lines $\emph{aff}(a_1, b_1)$ and $\emph{aff}(a_2, b_2)$ are skew then $C$ is indecomposable.
    \item Any simple extension of an indecomposable geometric graph is also indecomposable.
	\end{enumerate}
\end{theorem}




\section{Conditional decomposability}

We begin by recalling the wedge construction,  as it gives most easily the existence of conditionally decomposable polytopes in higher dimensions. 

Let $P$ be a $d$-polytope, expressed as an intersection of closed half-spaces, $P=\{x\in\r^d:Ax\ge 0\}$, where each $A$ is an affine function taking values in $\r^d$ and the inequality is meant coordinate-wise. Let $F$ be a nonempty proper face of $P$  supported by an affine function $f$, i.e. $f(x)\ge0$ for all $x\in P$ and $F=\{x\in P:f(x)=0\}$. Then the {\it wedge} of $P$ at $F$ is the $(d+1)$-polytope $W(P,F)=\{(x,t)\in\r^{d+1}:Ax\le b, 0\le t\le f(x)\}$. Two of the facets of $W(P,F)$ are copies of $P$, one called the lower base in the hyperplane $t=0$, and one called the upper base in the hyperplane $t=f(x)$. All vertices of the wedge lie in either the upper base or the lower base, and their intersection is just $F$. The other facets of $W(P,F)$ are 
wedges that contain faces of $F$, and prisms that are disjoint from $F$. For further details, refer to \cite[Proposition 1.4]{KleeWalkup1967}  or \cite[Section 2.5]{Pineda-2023}.

The next lemma gives a fundamental relationship between wedges and Minkowski sums. Recall that for a given normal vector, the corresponding face of the Minkowski sum is the sum of the corresponding faces of the summands.

\begin{lemma}\label{wedge_sum}
Let $P_1$ and $P_2$ be $d$-polytopes with nonempty faces $F_1$ and $F_2$, which 
are determined by the same normal vector. Then
$$ W(P_1+P_2, F_1+F_2) = W(P_1, F_1) + W(P_2, F_2). $$
\begin{proof}
Denote by ${P_i}$ the lower base of $W(P_i, F_i)$, and by $\overline{P_i}$ the upper base of $W(P_i, F_i)$. It is clear that $P_1+P_2$ and $\overline{P_1}+\overline{P_2}$ are facets of $W(P_1, F_1) + W(P_2, F_2)$. Their intersection will be a face determined by the common normal vector of  $F_1$ and $F_2$, i.e. $F_1+F_2$. This establishes that $W(P_1, F_1) + W(P_2, F_2)$ contains $W(P_1+P_2, F_1+F_2)$. 

Conversely, we need to show that every vertex of the Minkowski sum lies in the wedge. It suffices to show that the sum of a vertex of $W(P_1, F_1)$ and a vertex of $W(P_2, F_2)$ lies in  $W(P_1+P_2, F_1+F_2)$. This is clear if both vertices lie the upper base or if both vertices lie the lower base. We are left with the case of a vertex in the lower base of $W(P_1, F_1)$ and a vertex in the upper base of $W(P_2, F_2)$. These may be expressed in the form $(v_1,0)$, where $v_1\in P_1\setminus F_1$, and $(v_2,f(v_2))$, where $v_2\in P_2\setminus F_2$. But the sum of these two is a convex combination of $(v_1+v_2,0)\in P_1+P_2$ and $(v_1+v_2,f(v_1)+f(v_2))\in\overline{P_1}+\overline{P_2}$.  
\end{proof}
\end{lemma}

\begin{theorem}\label{wedge_thm}
  Let $P$ be a $d$-polytope and $F$ a nonempty proper face. Then  
  \begin{enumerate}
    \item $P$ is decomposable if and only if $W(P,F)$ is decomposable; $P$ is indecomposable if and only if $W(P,F)$ is indecomposable. 
    \item if $P$ is conditionally decomposable, then $W(P,F)$ is conditionally decomposable.
    \item for $d\ge3$ and $f_0\ge4d-4$, there is a conditionally decomposable $d$-polytope with $f_0$ vertices.
  \end{enumerate}
  \begin{proof}
    (1) First suppose $P$ is decomposable, say $P=P_1+P_2$, where $P_1$ and $P_2$ are not homothetic to one another. Denote by $F_1$ and $F_2$ the faces of $P_1$ and $P_2$ with the same outer normal as $F$; then $F=F_1+F_2$. Now $W(P_1, F_1)$ and $W(P_2, F_2)$ cannot be homothetic to one another because the facets corresponding to  $P_1$ and $P_2$ are not. By \cref{wedge_sum}, $W(P, F)=W(P_1, F_1)+W(P_2, F_2)$ is decomposable.
    
    Conversely, if $P$ is indecomposable, the graph of $P$ is indecomposable and touches every facet of $W(P,F)$, making $W(P,F)$  indecomposable.

    (2) If $P$ has both decomposable and indecomposable realisations, so does $W(P,F)$.
    
    (3) For $d=3$, we know that there is a conditionally decomposable 3-polytope with $4d-4$ vertices, $d+5$ facets, one facet $F$ which contains all but four of the vertices, and a simplex facet which meets $F$ in a ridge. Forming the wedge over $F$ gives us a 
    polytope with dimension one higher, 4 more vertices, and one more facet, and two facets with the properties just specified. Repeating, we have the existence of conditionally decomposable $d$-polytopes with $4d-4$ vertices and $d+5$ facets. In the lower dimensional base at each step, there is a simplex facet that intersect with $F$ at a ridge, so in the $W(P,F)$, there will be a simplex facet whose intersection with $P$ is a ridge. So it gives a simplex facet in the wedges constructed this way. Since gluing a pyramid over the simplex facet does not change the decomposability \cite[Proposition 6]{Przeslawski-Yost-2016}, it gives the existence of conditionally decomposable $d$-polytope with any number of vertices from $4d-4$ onwards.
  \end{proof}
\end{theorem}

\begin{rmk}
In an earlier version of this manuscript, we constructed examples with explicit coordinate representations. It was pointed out by an anonymous referee that most of these examples were wedges. Taking this into account has substantially improved the quality of this paper. In particular, \cref{wedge_sum}, \cref{wedge_thm}(2) were observed by this referee.
\end{rmk}

One special case of \cref{wedge_thm} is having $F282$ or $F288$ as a face, which were our original examples. When $d=4$, this gives two conditionally decomposable 4-polytopes with $f$-vector $(12,28,23,9)$, with the decomposable version having a line segment for a summand. There is a third conditionally decomposable 3-polytope with 8 vertices and 8 facets, namely the gyrobisfastigium $F287$, which is the sum of two triangles. Applying \cref{wedge_thm} in this case, we obtain another conditionally decomposable 4-polytope with $f$-vector $(12,28,23,9)$, but without a line segment for a summand.



Several other conditionally decomposable polytopes are of interest. It will be convenient to recall some notation: $\Delta(m,n)$  denotes as usual the Minkowski sum of an $m$-dimensional simplex and an $n$-dimensional simplex, lying in complementary subspaces of $\mathbb{R}^{m+n}$, or any polytope combinatorially equivalent to it. Similarly one defines  $\Delta(l,m,n)$. Then $\Delta(1,2)$ is the familiar triangular prism and $\Delta(1,1,1)$ is the cube. More generally $\Delta(1,d-1)$ is a prism based on a simplex; we will refer to it henceforth simply as a {\it prism}, or a $d$-prism if the dimension needs to be specified. 

\begin{remark}\label{rem:4d-2}
    There is a simpler construction available for a conditionally decomposable $d$-polytope with $4d-2$ vertices, which naturally generalizes Kallay's example. 
    \begin{proof}
    Start with a $d$-polytope $Q$ of the form $\Delta(1,1,d-2)$ with vertices
	\begin{equation*}
	\begin{aligned}
	& A_i=e_i, i=1,\dots, d-2, \\
	& A_{d-1}=(0,\dots,0), \text{ the origin}, \\
	& B_i=A_i+(0,\dots,0,1,0), i=1,\dots, d-1, \\
	& C_i=A_i+(0,\dots,0,0,1), i=1,\dots, d-1, \\
	& D_i=A_i+(0,\dots,0,1,1), i=1,\dots, d-1.
	\end{aligned}
	\end{equation*}
    Perturb $Q$ first, replacing $C_i$ with $C_{i}'=A_i+(0,\dots,0,-\varepsilon,3)$, $i=1,\dots,d-1$, and retaining the rest of the vertices. This yields a polytope $Q'$ which is  combinatorially equivalent to $Q$ but with some skew lines. Next we stack appropriate pyramids over two facets of $Q$, namely $\text{conv} \{C_i, D_i: i=1,\dots, d-1 \}$ and $\text{conv} \{A_i, B_i: i=1,\dots, d-1 \}$ to obtain a new polytope $P$. Having a segment for a summand, $P$ is decomposable. Then stack appropriate pyramids over the corresponding facets of $Q'$, i.e. $\text{conv} \{C_{i}', D_i: i=1,\dots, d-1 \}$ and $\text{conv} \{A_i, B_i: i=1,\dots, d-1 \}$,  to obtain a new polytope $P'$. Since the geometric graph of a pyramid is indecomposable, we obtain that $P'$ is indecomposable although it has the same face lattice as $P$ \cite[Proposition 6]{Przeslawski-Yost-2016}. Thus we have an alternative construction of conditionally decomposable polytopes with $4d-2$ (or more) vertices.

    \end{proof}
\end{remark}

Figure 1 shows the skeleton of some examples when $d=4$.

\begin{figure}[H]
	\begin{tabular}{ccc}
		\includegraphics[width=0.3\textwidth]{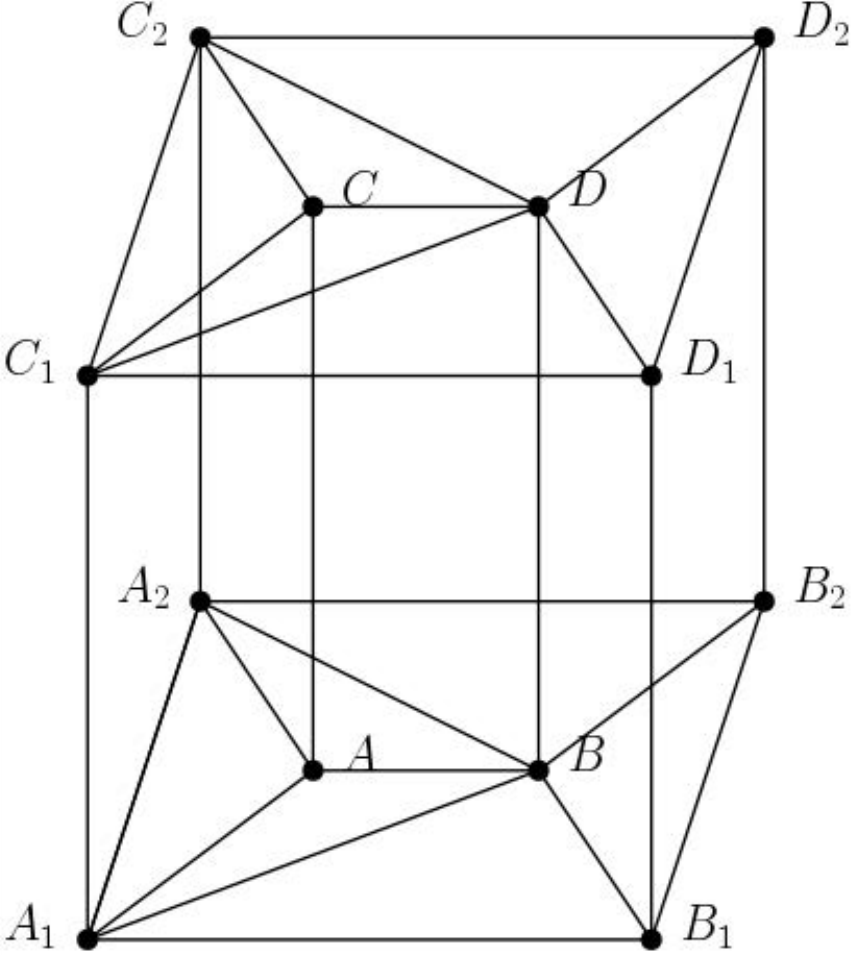} &
		\includegraphics[width=0.3\textwidth]{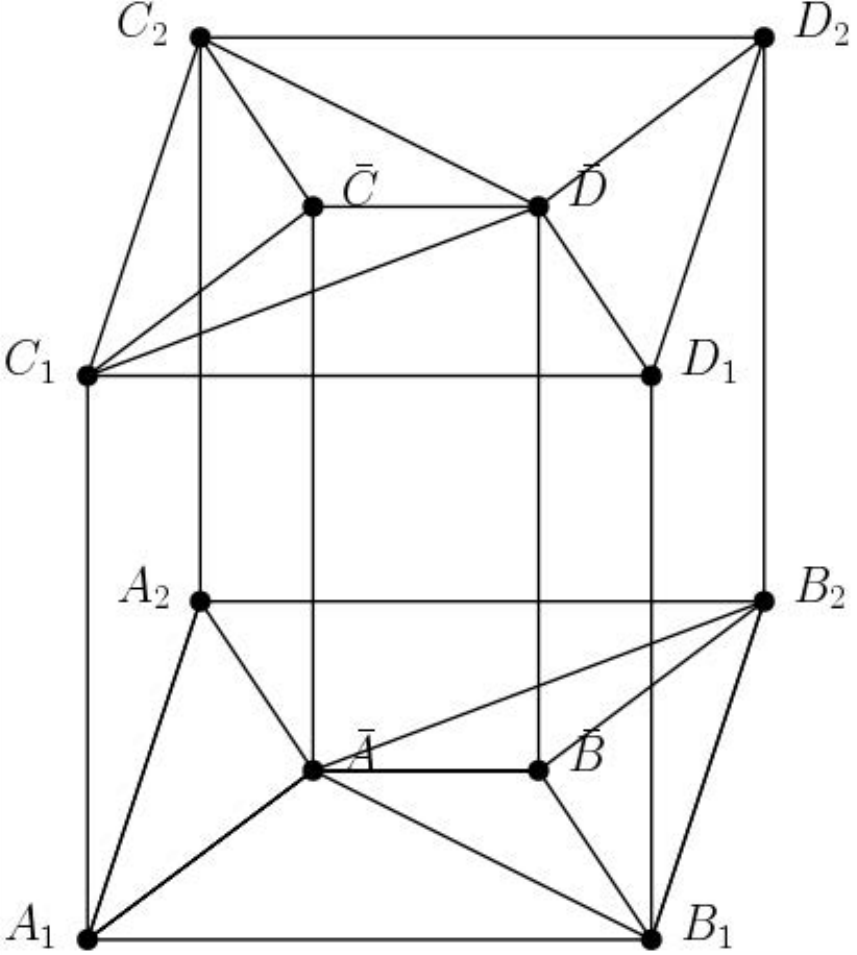}&
		\includegraphics[width=0.3\textwidth]{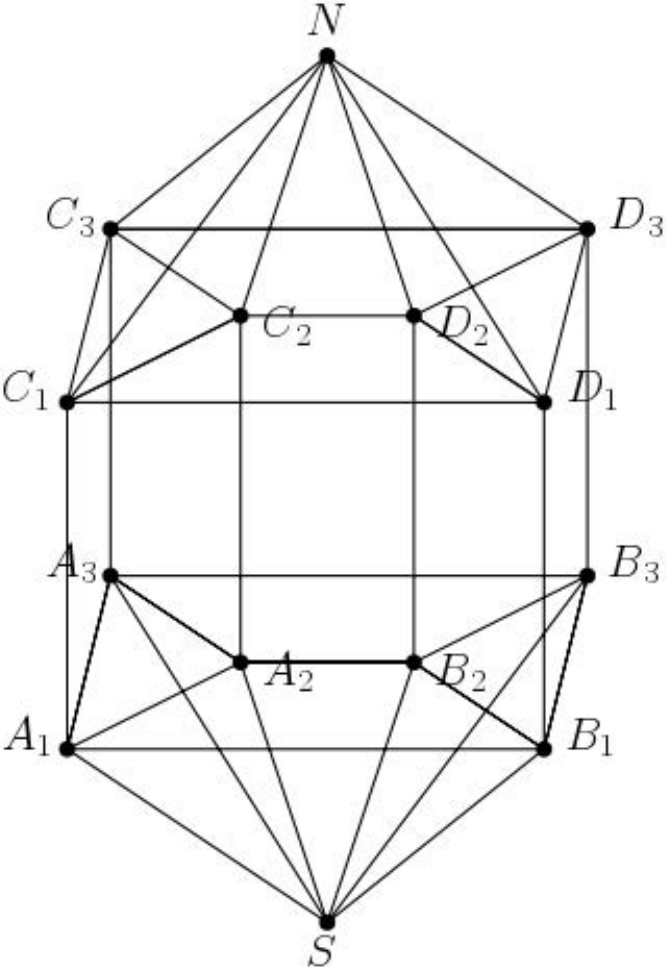}\\
		(1) two facets  $F282$ & (2) two facets  $F288$  &  (3) $P$ has $4d-2$ vertices
	\end{tabular}
	\caption{$4$-dimensional examples}
\end{figure}



Another way to look at these examples is to consider them as the convex hull of the union of three planar quadrilaterals, with different arrangements providing more examples. Variants of this leads to more example. For simplicity, we present just one more, which is not covered by the previous results. Amongst its facets is $F36$, which is a decomposable polyhedron, obtained by perturbing one vertex of a cube. 

\begin{example}\label{ex:last}
    There is a conditionally decomposable 4-polytope with 12 vertices and 9 facets, namely two  prisms, two  copies of $F36$, two simplices, two square pyramids and one of the form $F282$. In particular, it has only one conditionally decomposable facet.
    \begin{proof}
    Let $A_1=(1,0,0,0)$, $B_1=(1,1,0,0)$, $C_1=(1,0,2,0)$, $D_1=(1,1,1,0)$,  $A_2=(0,0,0,0)$, $B_2=(0,1,0,0)$, $C_2=(0,0,2,0)$, $D_2=(0,1,2,0)$, $A_3=(0,0,1,1)$, $B_3=(0,1,0,1)$, $C_3=(0,0,2,1)$, $D_3=(0,1,2,1)$. It is not hard to verify that these are the vertices of a polytope whose facets are   two  prisms $A_1 A_2 A_3 C_1 C_2 C_3$ and $B_1 B_2 B_3 D_1 D_2 D_3$, two copies of $F36$ namely $A_i B_i C_i D_i A_2 B_2 C_2 D_2$ for $i=1$ and $i=3$, two simplices $A_1 A_2 A_3 B_3$ and $C_1 D_1 D_2 D_3$, two square pyramids $A_1 A_2 B_1 B_2 B_3$ and  $C_1 C_2 C_3 D_2 D_3$ and one copy of $F282$, $A_1 B_1 C_1 D_1 A_3 B_3 C_3 D_3$. By perturbing the vertices $C_1,C_2,C_3$ as before, this can be shown to be conditionally decomposable. 
    \end{proof}
\end{example}


\section{The number of vertices}

It is natural to ask whether $4d-4$ is a sharp lower bound for the number of vertices of conditionally decomposable polytopes. We now show that among examples which have a line segment for a summand, it is the best possible.

Notice that if $P$ has a line segment for a summand, then the sums of $P$ with this line segment with different length are all combinatorially equivalent, since the summands would lie in the relative interior of the cone of summands of $P$ which are all locally similar, and hence combinatorially equivalent \cite[Corollary 2.7]{Meyerthesis}.

The following lemma is an elementary exercise.

\begin{lemma}\label{lem:easy} If three lines in $\mathbb{R}^d$ are pairwise coplanar, but not mutually coplanar, then they are either parallel or concurrent.
\end{lemma}

\begin{theorem}\label{thm:4d-5}
    Let $P$ be a $d$-polytope with no more than $4d-5$ vertices which has a line segment for a summand. Then $P$ is combinatorially decomposable.	
\end{theorem}
\begin{proof}
	Assume that $P=Q+[0,c]$ for another polytope $Q$ and vector $c$. Denote by $A$ the set of vertices of $P$ which are also vertices of $Q$, and by $B$ the set of vertices of $P$ of the form $v+c$ for some vertex $v$ of $Q$.
	
	By replacing $c$ if necessary by a suitable multiple of itself, we can ensure that
	$$\min\{b\cdot c:b\in B\}=\min\{v\cdot c+\|c\|^2:v+c \text{ is a vertex of } P\text{ and }v\text{ is a vertex of } Q\} $$
	is as large as we like, in particular, larger than $\max\{a\cdot c:a\in A\}$. This means that we can find a hyperplane $H$ (normal to the direction $c$) which strictly separates $A$ and $B$. Clearly $H$ contains an interior point  of each edge of $P$ parallel to $[0,c]$ (and not of any other edges). 
	
	Of the two open half-spaces associated with $H$, clearly one contains $A$ and one contains $B$. Every edge of $P$ with one vertex in $A$ and one vertex in $B$ will be parallel to $[0,c]$. Any other edge of $P$ (i.e. with both vertices in $A$ or both vertices in $B$) will not be parallel to $[0,c]$. 
	
	Now consider the $(d-1)$-polytope $C= P \cap H$. Every vertex of $C$ is an interior point of a unique edge of $P$ parallel to $[0,c]$. Every edge of $C$ is contained in a unique 2-face of $P$, determined by the two edges of $P$ through its vertices. Every triangular 2-face of $C$ is the intersection of $H$ with a unique 3-face of $P$; this 3-face is determined by the three edges of $P$ through the vertices of the triangle.
	
	Since $P$ has no more than $4d-5$ vertices, at least one of $A,B$, let us say $A$, has no more than $2(d-1)-1$ vertices. The same must hold for $C$, and \cite[Lemma 5(ii)]{Yost1991} then informs us that there is a strongly connected triangular chain of faces containing all the vertices of $C$.
	
	In terms of $P$ this means that there is a finite sequence $S_1,S_2,\ldots$ of triples of edges of $P$ such that
	\begin{enumerate}
		\item for each such triple $S=\{e_1,e_2,e_3\}$, there is a unique 3-face of $P$ containing the three edges, and each two of them determine a 2-face, 
		\item each successive two such sets $S_i$, $S_{i+1}$ have two common edges; in particular, the respective 3-faces intersect in a 2-face of $P$,
		\item every edge with one vertex from $A$ and one vertex from $B$ belongs to at least one $S_i$.
	\end{enumerate}
	
	Now consider a polytope $P'$ combinatorially equivalent to $P$, and denote by $A'$ and $B'$ the collections of vertices corresponding to $A$ and $B$. Then $P'$ will have a finite sequence $S_1,S_2,\ldots$ of triples of edges, each with one vertex from $A'$ and one vertex from $B'$, such that
	\begin{enumerate}
		\item every edge with one vertex from $A'$ and one vertex from $B'$ belongs to at least one $S_i$,
		\item for each such $S=\{e_1,e_2,e_3\}$,  each two edges determine a 2-face and so are coplanar, but the three of them  are not coplanar,
		\item each successive two such sets $S_i$, $S_{i+1}$ have two common edges; in particular, the respective 3-faces intersect in a 2-face of $P$,
		\item all edges parallel to $[0,c]$ belong to at least one $S_i$.
	\end{enumerate}
	
	Applying \cref{lem:easy}, we see that all of the edges between $A'$ and $B'$ are either parallel or contained in concurrent lines. In the latter case, the concurrent point will not belong to $P'$.  Taking a suitable projective transformation (i.e. one whose kernel contains the concurrent point without intersecting  $P$), we obtain a $d$-polytope combinatorially equivalent to $P$ with all those edges being parallel. In either case, we see that $P'$ is  combinatorially equivalent to a polytope with a line segment as a summand. The combinatorial decomposability of $P$ then follows from the fact that decomposability of a polytope is projectively invariant \cite[Theorem 3.6]{Kallay1984} (see also \cite[Application (f), p. 40]{Smilansky1987}).
\end{proof}	

The only decomposable polytope with $2d$ vertices is a simplicial prism (\cite{Przeslawski-Yost-2016} and the references therein), and it has line segments for a summand;  a decomposable polytope with $2d+1$ vertices with $d\ge5$ is either a pentasm or a capped prism \cite[Section 5]{PUY2018Excess} (again these have line segments for summands); for a decomposable polytope with $2d+2$ vertices, when $d\ge 6$, they are either 2-capped prism, or have a line segment for a summand, whose cross section is a $(d-1)$-polytope with $d+1$ vertices \cite{wjphdthesis}. 

Based on these, we suspect that when the decomposable polytope does not have too many vertices, it is combinatorially (even projectively) equivalent to a polytope having a line segment for a summand.

\begin{lemma}
    Let $P$ be a $d$-dimensional polytope having a line segment for a summand, and with $\le 3d-4$ vertices. Let $V_1$, $V_2$ be the partition of vertices of $P$ such that  each vertex in $V_1$ (respectively $V_2$) has at most one neighbor in $V_2$ (respectively $V_1$). Then the subgraphs of $P$ induced by $V_1$ and $V_2$ are indecomposable.
    \begin{proof}
    We proceed by induction on the dimension. The induction base is $d=4$, where $P$ is a simplicial prism, the subgraphs $G(V_1)$, $G(V_2)$ are the graph of 3-simplices, which are indecomposable.
		
    Suppose it is true for dimension $d-1$. Arguing as in \cref{thm:4d-5} we can consider a cross section of $P$ to show that there are at least $d$  ``side facets" of $P$ which each have a line segment for a summand. Let $F$ be one of them, with partition of vertices $V_1(F) \subset V_1$, $V_2(F) \subset V_2$. Then $F$ has $[2(d-1), 3(d-1)-3]$ possible  vertices. 
    
    Suppose $F$ has between $2(d-1)$ and $3(d-1)-4$  vertices: then the subgraphs induced by $V_1(F), V_2(F)$ are indecomposable, respectively, by induction. There are $\le d-2$ vertices outside $F$. Let $v$ be one of them and assume without loss of generality that $v\in V_1$. Note that $v$ has to be adjacent to at least three vertices of $V_1(F)\cup V_2(F)$. Since $v$ is adjacent to at most one vertex of $V_2$, $v$ is adjacent to at least two vertices of $V_1(F)$, so $v$ is a simple extension of $G(V_1(F))$, the subgraph induced by $V_1(F) \cup \{v\}$ is indecomposable. 
    
    The remaining case is that $F$ has $3(d-1)-3$ vertices and there are exactly 2 vertices $x,y$ outside $F$, then no vertex in $F$ is adjacent to both $x,y$, i.e. each vertex in $V_1(F)$ is adjacent (and only) to (say) $x$, and each vertex in $V_2(F)$ is adjacent (and only) to $y$. Now again consider a hyperplane $H$ separating $V_1$ from $V_2$. The intersection of $F$ with the corresponding half-spaces will be two $(d-1)$-polytopes with a common facet, and all other vertices will be in $V_1(F)$ or $V_2(F)$. Removing a facet from any polytope leaves us with a connected graph. So the subgraph induced by $V_1(F)$ is connected, i.e. there is a polyhedral path connecting all the vertices in $V_1(F)$, and thus there is a strongly connected triangular chain of faces in the subgraph induced by $V_1(F)\cup\{x\}$ connecting each vertex. Hence $G(V_1)$ is indecomposable, and the same holds for $G(V_2)$.
    \end{proof}
\end{lemma}

\begin{theorem}\label{thm_linesegment}
    Let $P$ be a decomposable $d$-polytope with $[2d, 3d-4]$ vertices, then $P$ has a line segment for a summand with respect to projective equivalence.
    \begin{proof}
    We proceed by induction. The induction base is $d=4$ which is correct.
    Suppose it is true for $d-1$. By \cref{theorem-indecomposable}  we can find  a decomposable facet $F$ of $P$. (In fact, $P$ must have at least $d$  decomposable facets \cite[Theorem 17]{PUY2018Excess}.) The number of vertices of $F$ is in the range $[2(d-1), 3d-6]$; by \cref{theorem-indecomposable}(1) we can exclude the case that $P$ is a pyramid. 
 		
    Suppose first that there is a facet $F$ with $3d-6$ vertices. Then there are exactly two vertices $u, v$ outside $F$ which are adjacent. Then every vertex in $F$ has exactly one edge outside $F$, otherwise there is a triangle that touches every facet of $P$. We claim that the adjacency of vertices of $F$ with the vertices of $u, v$ defines a decomposing function $f$ for $F$, namely, if $x\in V(F)$ and $x$ is adjacent to $u$, then $f(x)=u$. We may assume that the line containing $[u,v]$ is parallel to $\text{aff}\{F\}$. It suffices to show that if $x\in V(F)$ is adjacent to $u$, and $y\in V(F)$ is adjacent to $v$, and $x, y$ are adjacent, then $[x, y]$ is parallel to $[u, v]$. Since there must exist a 2-face containing $[x, y]$ but not contained in $F$, it follows that the quadrilateral $x, y, v, u$ is a 2-face. So the lines containing the four edges are coplanar, hence $[u,v]$ is parallel to $[x, y]$. By defining $f(u)=u, f(v)=v$, it defines a decomposing function for $P$. It then shows that $P$ has a line segment for a summand.
		
    \begin{figure}[H]
		\includegraphics[width=0.35\textwidth]{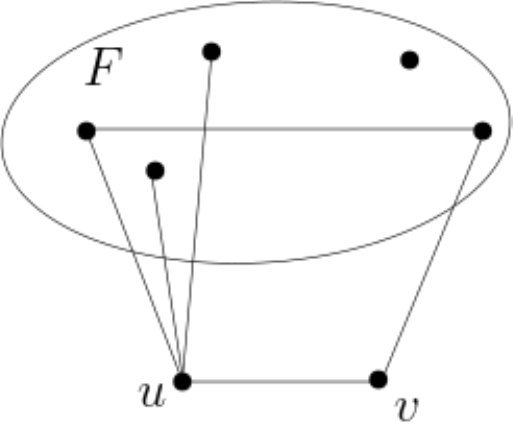}
    \end{figure}
 		
    Assume that the number of vertices of a decomposable facet $F$ is in the range $[2(d-1), 3(d-1)-4]$. By induction, $F$ has a line segment for a summand. Without loss of generality, we may assume that the side edges of $F$ are parallel.  Denote the partition of vertices of $F$ by $V_1(F), V_2(F)$. There are at most $d-2$  vertices outside $F$. So each such vertex $v$ will be adjacent to at least three vertices in $F$. \\
     If $v$ is only adjacent to one side of vertices of $F$, then it is a simple extension, making the union an indecomposable subgraph.\\
    Otherwise, $v$ is adjacent to both sides of vertices of $F$, say it has at least two neighbors in $V_1(F)$, and at least one neighbor $v'$ in $V_2(F)$. This implies that subgraph $G(V_1(F)\cup v)$ is indecomposable, and that the line containing $v, v'$ is skew to almost every such side edge. Now we have indecomposable subgraphs connected by skew lines, so $G(V_1(F)\cup \{v\} \cup V_2(F))$ is indecomposable. There are at most $d-3$ vertices not included in this subgraph, which implies that it is not enough to form a facet, hence $G(V_1(F)\cup \{v\} \cup V_2(F))$ is a subgraph which touches every facet of $P$, making $P$ indecomposable by \cref{theorem-indecomposable}, which is a contradiction.\\
    So, every vertex outside $F$ is a simple extension of either $G(V_1(F))$ or $G(V_2(F))$. (In fact, we can claim that there must be an edge outside $F$ that is parallel to its side edges. This can be seen by looking at the other facet $F'$ that intersects with $F$ at a ridge, where $F'$ has a line segment for a summand, by induction.) \\
    Now we have two indecomposable subgraphs, what remains to show is that any edge connecting them is parallel to the side edges of $F$, which is a routine discussion. Hence we can define a decomposing function whose image is a line segment. 
    \end{proof}
\end{theorem}

Note that $\Delta(2,d-2)$ is decomposable, has $3d-3$ vertices and does not have a line segment for a summand. So the previous theorem cannot be extended further.

\begin{corollary}
    There is no conditionally decomposable polytope with $[2d, 3d-4]$ vertices.
\end{corollary}


\section{A lower bound result}

Recall Gr{\"u}nbaum's lower bound theorem \cite{Grunbaum2003}, that was recently proved by Xue \cite{Xue2021} that, for a $d$-polytope with $d+m$ vertices, $m\le d$, the number of $k$-faces of the polytope is at least 
$$\phi_k(d+m,d) =\binom{d+1}{k+1}+\binom{d}{k+1}-\binom{d+1-m}{k+1}.$$
The minimiser is a $(d-m)$-fold pyramid over a $m$-dimensional simplicial prism, denoted by $M(k,d-k)$. 

The lower bound of the number of $k$-faces of $d$-polytopes with $2d+1$ vertices was established by Pineda-Villavicencio, Ugon, the second author \cite{2dplus1}, and by Xue \cite{Xue2dplus1}, independently. The lower bound on the number of edges of $d$-polytopes with $2d+2$ vertices was established by Pineda-Villavicencio and the second author \cite{2dplus2}. It is conjectured in \cite[Section 3]{2dplus2} that for $d$-polytopes with $2d+m \in [2d+3, 3d-6]$ vertices, the lower bound of the number of edges (and perhaps the $k$-faces) is reached by the polytope resulting from truncating a simple vertex of $M(m+1,d-m-1)$, denoted by $M(m+1,d-m-1)^{\text{cut}}$. Note that 
$$f_k(M(m+1,d-m-1)^{\text{cut}})=\binom{d+1}{k+1}+2\binom{d}{k+1}-\binom{d-m}{k+1}.$$ 

We will establish a partial answer, that $M(m+1,d-m-1)^{\text{cut}}$ reaches the lower bound amongst the decomposable polytopes.
 
Rcall Xue's sequence \cite{Xue2021} that the number of $k$-faces of a $d$-polytope containing some vertices $v_2,\dots, v_l$ but not $v_1$ is at least 
$$\sum_{i=2}^{l}\binom{d+1-i}{k}.$$

\begin{lemma}
    Let $P$ be a $d$-polytope with $d+m$ vertices, $2 \le m\le d$, and let $v$ be a vertex of $P$. We have the following, \\
    (i) If $P$ is a pyramid and $v$ is an apex of $P$, where $F$ is the unique facet that does not contain $v$, then $\# k$-faces of $P$ that does not contain $v$ is $f_k(F)$;\\
    (ii) If $P$ is not a pyramid, then $\# k$-faces of $P$ that does not contain $v$ is at least $\phi_k(d+m,d)-\binom{d}{k}$.
\end{lemma}
\begin{proof}
 (i) This is rather obvious, as a face of $P$ is either a face containing $v$ or a face of $F$. And when $m\le d-1$,
    \begin{equation*}
       \begin{aligned}
           f_k(F) & \ge \phi_k(d+m-1, d-1) \\
                  & = \binom{d}{k+1}+\binom{d-1}{k+1}-\binom{d-m}{k+1} \\
                  & =\phi_k(d+m,d)-\binom{d}{k}-\binom{d-1}{k}+\binom{d-m}{k}.
       \end{aligned}
    \end{equation*}
 (ii) Let $F$ be a facet not containing $v$, and the number of vertices of $F$ is $d+s$, $0 \le s \le m-2$. By Xue's sequence,
 \begin{equation}
     \begin{aligned}
         \# k\text{-faces of } P & \ge f_k(F)+\sum_{i=2}^{m-s}\binom{d+1-i}{k} \\
         & \ge \phi_k(d+s,d-1)+\sum_{i=2}^{m-s}\binom{d+1-i}{k} \\
         & = \binom{d}{k+1}+\binom{d-1}{k+1}-\binom{d-s-1}{k+1}+\sum_{i=2}^{m-s}\binom{d+1-i}{k} \\
         & =\phi_k(d+m,d)-\binom{d}{k}-\binom{d-1}{k}+\binom{d+1-m}{k+1}-\binom{d-s-1}{k+1}+\sum_{i=2}^{m-s}\binom{d+1-i}{k} \\
         & =\phi_k(d+m,d)-\binom{d}{k}-\binom{d-1}{k}+\sum_{i=0}^{d-m}\binom{i}{k}-\sum_{i=0}^{d-s-2}\binom{i}{k}+\sum_{i=d-m+s+1}^{d-1}\binom{i}{k} \\
         & =\phi_k(d+m,d)-\binom{d}{k}-\binom{d-1}{k}-\sum_{i=d-m+1}^{d-s-2}\binom{i}{k}+\sum_{i=d-m+s+1}^{d-1}\binom{i}{k} \\
         & \ge \phi_k(d+m,d)-\binom{d}{k},
     \end{aligned}
 \end{equation}
  where the last equality holds if and only if $s=0$, i.e. $F$ is a simplex facet.  
\end{proof}

\begin{theorem}
    Let $P$ be a decomposable $d$-polytope with $2d+m$ vertices, where $m\in [0, d-4]$. Then $f_k(P) \ge \binom{d+1}{k+1}+2\binom{d}{k+1}-\binom{d-m}{k+1}$, for all $k$. When $m=2$, there are exactly two minimizers, the $M(3,d-3)^{\text{cut}}$ and the prism based on $M(2,d-3)$; when $m\ne 2$, there is a unique minimizer, the $M(m+1,d-m-1)^{\text{cut}}$. 
    \begin{proof}
           By \cref{thm_linesegment}, we may assume that $P$ has a line segment for a summand. Then there is a corresponding partition of the vertices $V=V_1\cup V_2$, with the edges between $V_1$ and $V_2$ all parallel. Take a suitable projective transformation $T_1$ such that the corresponding concurrent lines in $P_1=T_1(P)$ intersect in a half space containing $T_1(V_2)$ but not $T_1(V_1)$. Similarly, we can find another projective transformation $T_2$ such that the corresponding concurrent lines in $P_2=T_2(P)$ intersect in a half space containing $T_2(V_1)$ but not $T_2(V_2)$.
    \begin{figure}[H]
	\begin{tabular}{cc}
		\includegraphics[width=0.4\textwidth]{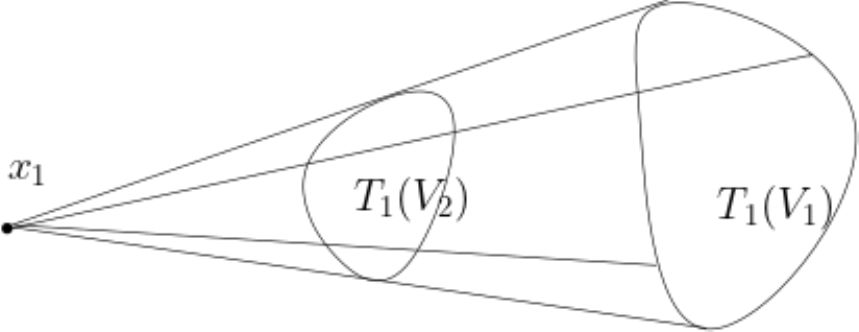}&
		\includegraphics[width=0.4\textwidth]{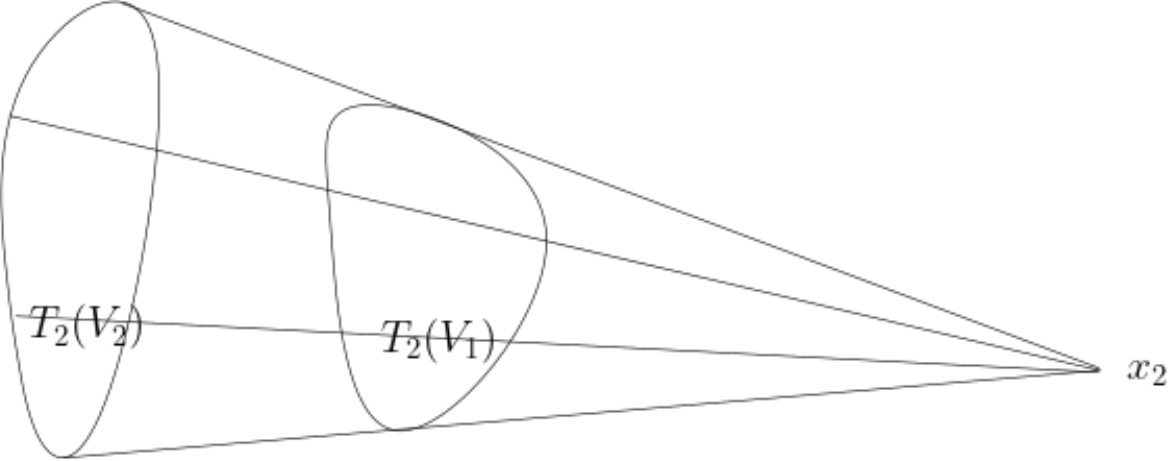} \\
		(1) $P_1$  & (2) $P_2$ 
	\end{tabular}
    \end{figure}
   The number of $k$-faces of $P$ is exactly the number of $k$-faces of $P_1$ (or $P_2$, respectively) plus the number of $k$-faces of $P_2$ (or $P_1$, respectively) that does not contain $x_2$ (or $x_1$, respectively). 
   
   Suppose $P_2$ is not a pyramid. We may assume that $P_1$ has $d+s$ vertices, and $P_2$ has $d+m-s+2$ vertices, $s\ge 1$ and $m-s \ge -1$. Then 
       \begin{equation}
           \begin{aligned}
               f_k(P) & \ge \phi_k(d+s,d)+\phi_k(d+m-s+2,d)-\binom{d}{k} \\
               & =\binom{d+1}{k+1}+\binom{d}{k+1}-\binom{d+1-s}{k+1}+\binom{d+1}{k+1}+\binom{d}{k+1}-\binom{d-m+s-1}{k+1}-\binom{d}{k} \\
               & = \binom{d+1}{k+1}+2\binom{d}{k+1}-\binom{d-m}{k+1}+\binom{d-m}{k+1}-\binom{d+1-s}{k+1} \\
               & \quad +\binom{d+1}{k+1}-\binom{d-m+s-1}{k+1}-\binom{d}{k} \\
               & =\binom{d+1}{k+1}+2\binom{d}{k+1}-\binom{d-m}{k+1}+\sum_{i=0}^{d-m-1}\binom{i}{k}+\sum_{i=0}^{d}\binom{i}{k} \\
               & \quad -\sum_{i=0}^{d-s}\binom{i}{k}-\sum_{i=0}^{d-m+s-2}\binom{i}{k}-\binom{d}{k} \\
               & =\binom{d+1}{k+1}+2\binom{d}{k+1}-\binom{d-m}{k+1}-\sum_{i=d-m}^{d-m+s-2}\binom{i}{k}+\sum_{i=d-s+1}^{d-1}\binom{i}{k} \\
               & \ge \binom{d+1}{k+1}+2\binom{d}{k+1}-\binom{d-m}{k+1},
           \end{aligned}
       \end{equation}
    where the last equality holds if and only if $s=1$, i.e. the partition of one side of the vertices forms a simplex facet, and there are exactly $d$ edges in between the partition of the vertices. One condition for the first equality to hold is that $P_2$ is the minimizer with $d+m+1$ vertices, which is $M(m+1,d-m-1)$. Hence, $M(m+1,d-m-1)^{\text{cut}}$ is the unique minimizer in this case. \\
    The remaining case is that $P_1$, $P_2$ are both pyramids, and $x_1$, $x_2$ are apexes, respectively. It follows that $\text{conv}(V_1)$, $\text{conv}(V_2)$ are facets of $P$ and $V_1$, $V_2$ have the same number of vertices. We may assume that $P$ has $2d+2m$ vertices, $1 \le m \le (d-4)/2$. Then 
    \begin{equation}
        \begin{aligned}
            f_k(P) & \ge 2\phi_k(d+m, d-1)+\phi_{k-1}(d+m, d-1) \\
            & = 2\binom{d}{k+1}+2\binom{d-1}{k+1}-2\binom{d-m-1}{k+1}+\binom{d}{k}+\binom{d-1}{k}-\binom{d-m-1}{k} \\
            & =\binom{d+1}{k+1}+2\binom{d}{k+1}-\binom{d-2m}{k+1}+\binom{d-2m}{k+1}+\binom{d-1}{k+1}-2\binom{d-m-1}{k+1}-\binom{d-m-1}{k} \\
            & = \binom{d+1}{k+1}+2\binom{d}{k+1}-\binom{d-2m}{k+1}+\sum_{i=0}^{d-2m-1}\binom{i}{k}-2\sum_{i=0}^{d-m-2}\binom{i}{k}+\sum_{i=0}^{d-2}\binom{i}{k}-\binom{d-m-1}{k}  \\
            & =\binom{d+1}{k+1}+2\binom{d}{k+1}-\binom{d-2m}{k+1}-\sum_{i=d-2m}^{d-m-2}\binom{i}{k}+\sum_{i=d-m-1}^{d-2}\binom{i}{k}-\binom{d-m-1}{k} \\
            & \ge \binom{d+1}{k+1}+2\binom{d}{k+1}-\binom{d-2m}{k+1},
        \end{aligned}
    \end{equation}
    where the last equality holds if and only if $m=1$, and the first equality holds if and only if $\text{conv}(V_1)$, $\text{conv}(V_2)$ are both $M(m+1, d-m-2)$, i.e. the unique minimizer for this case is the prism based on  $M(2, d-3)$.
    \end{proof}
\end{theorem}


We expect that the $d$-polytopes with between $2d+2$ and $3d-6$ vertices and minimum number of $k$-faces must be decomposable. However for $d>7$, a pyramid over $\Delta(2,d-3)$ is indecomposable and has strictly fewer edges than any decomposable polytope with $3d-5$ vertices \cite[Section 3]{2dplus2}.


\section{Lower bounds for the number of facets}

The results of Section 2 show that there are conditionally decomposable polytopes with $d+5$ facets, for every $d\ge3$. Now we investigate whether this bound can be reduced.

The following result is essentially contained in the second paragraph of \cite[p 88]{McMullen1973}, and formulated explicitly in \cite[Theorem 5.1]{Castillo-et-al}.
For a historical account of the evolution of this result and more profound details, see \cite[Theorem 5.1]{Castillo-et-al}, also \cite{FillastreIzmestiev2017}, or the chamber cone discussed in \cite{triangulations}.

\begin{theorem}\label{Gale-polar} (McMullen) 
Let $P$ be a polytope, $P^{\circ}$ the polar of $P$, and $Gale(P^{\circ})$ the Gale diagram of the polar of $P$. Then $P$ is decomposable if and only if $\dim(\bigcap S_i)>0$, and $P$ is indecomposable if and only if $\bigcap S_i$ is the origin, where $S_i$ ranges over all the cofacets of $P^{\circ}$. 
\end{theorem}

A $d$-polytope with $d+2$ facets is either simple or a pyramid, so there are no conditionally decomposable examples. We show that the same is true for $d$-polytopes with $d+3$ facets, based on the fact that the colinear/coplanar property in a Gale diagram is a combinatorial property.

\begin{proposition}
 There is no conditionally decomposable $d$-polytope with $d+3$ facets.
 \begin{proof}
    Let $P$ be a $d$-polytope with $d+3$ facets. Correspondingly,  $P^{\circ}$ is a $d$-polytope with $d+3$ vertices, whose Gale diagram is in $\mathbb{R}^2$ with $d+3$ points. Any cofacet in this Gale diagram is either a single point (if $P$ is a pyramid), two diametrically opposite points, or three points whose convex hull contains a neighborhood of the origin. The intersection of all such triangular cofacets will also contain a neighborhood of the origin, and thus have no essential effect on the intersection described in \cref{Gale-polar}. Thus it suffices to consider the cofacets of $Gale(P^{\circ})$ which are 1-dimensional. 
    
    If there are two different diameters, the intersection of these diameters will be the origin, then the intersection of all cofacets is the origin. And for any polytopes combinatorially equivalent to it, the two different diameters cannot coincide, so the intersection of all cofacets is the origin which is a combinatorial property, i.e. $P$ is combinatorially indecomposable. And if there are only identical diameters, those points remains colinear with respect to combinatorial equivalence, hence $P$ is combinatorially decomposable.
 \end{proof}
\end{proposition}

\begin{proposition}
    There is no conditionally decomposable 4-polytope with 8 facets.
    \begin{proof}
        The diagram $Gale(P^{\circ})$ is in $\mathbb{R}^3$, so it suffices to consider cofacets of dimension 1 or 2, i.e. cofacets that are line segments or triangles.

        Suppose there is a cofacet $F_1$ which is a diameter; then it has transforms of vertices at both end points. For the other cofacets which are line segments, if they are identical to $F_1$, they remains identical with respect to combinatorial equivalence, so they do not have any effect on decomposability. If they are different from $F_1$, then $P$ is combinatorially indecomposable. For the cofacets which are triangles, if they are coplanar with $F_1$, they remain coplanar; and if they are not coplanar, they remain not coplanar, with respect to combinatorial equivalence, i.e. if there is a triangular cofacet not coplanar with $F_1$, then $P$ is combinatorially indecomposable, and if all the triangular cofacets are coplanar with $F_1$, then $P$ is combinatorially decomposable.

        Now suppose that all the lower dimensional cofacets are triangles. Either the Gale diagram has two or fewer triangular cofacets, implying that their intersection is an interval whence $P$ is combinatorially decomposable, or it has at least 3 triangular cofacets. For the latter case, since there are exactly 8 points in the Gale diagram, some of the triangular cofacets have to share a common vertex, $v\in F_i \cap F_j$. If there is triangular cofacet $F$ that is not coplanar with $v$, then $F\cap F_i\cap F_j=\{0\}$, and then $P$ is combinatorially indecomposable. If all of the triangular cofacets are coplanar with $v$, then the intersection of all the cofacets contains $[0, v]$, and then $P$ is combinatorially decomposable.        
    \end{proof}
\end{proposition}

\begin{proposition}
 For all $d\ge5$, there is a conditionally decomposable $d$-polytope with $d+4$ facets.
 \begin{proof}
 It suffices to prove this when $d=5$; constructing wedges then establishes the conclusion in higher dimensions.

 In $\r^3$, it is not hard to find three triangles with vertices on the unit sphere, one each in the planes $z=0$, $x=0$ and $x=z$, each with the origin in its relative interior, such that no other triple of vertices has the origin in its relative interior, no two vertices are antipodal or coincident, and every open half-sphere contains at least two of them.
 Then we have 9 points which constitute the Gale diagram of a 5-polytope with 9 vertices. Its polar, which we will call $P$, is a 5-polytope with 9 facets. The three planes contain the $y$-axis, so \cref{Gale-polar}  
 ensures that $P$ is  decomposable. Indeed there are no other lower dimensional cofacets, so the intersection of all the cofacets is just the intersection of the three  triangular cofacets, which contains a segment in the $y$-axis.
 
 However if we tilt one the triangles slightly, the combinatorial structure of $P$ will not change, but the three planes, and hence the three triangles, will intersect only at the origin. This $P$ is conditionally decomposable.
 
\end{proof}
\end{proposition}

\section*{Acknowledgments} 
We would like to express our gratitude to an  anonymous referee for pointing out the structure of our first examples, i.e. that they are wedges, and for indicating in detail possible future lines of research.

%
%



Federation University Australia, P.O. Box 663, Ballarat, Victoria 3353, Australia\\\email{wangjiebulingbuling@foxmail.com} \and \email{d.yost@federation.edu.au}
\end{document}